%% file: inf-etc5.tex
\documentclass{article}

\usepackage{color}
\usepackage{graphicx}
\usepackage{stmaryrd}

\input{ab}

\input{thmlist}

\usepackage[numbers]{natbib}
\usepackage{natbibspacing}
\renewcommand{\bibfont}{\small}

\definecolor{myurlcolor}{rgb}{0.1,0.1,0.8}
\definecolor{mylinkcolor}{rgb}{0.05,0.05,0.4}
\usepackage{hyperref}
\hypersetup{colorlinks,
linkcolor=mylinkcolor,
citecolor=mylinkcolor,
urlcolor=mylinkcolor}

\title{Spaces of extremal magnitude}
\author{Tom Leinster%
\thanks{School of Mathematics, University of Edinburgh, Edinburgh EH9 3FD,
  Scotland; Tom.Leinster@ed.ac.uk.}%
\qquad
Mark Meckes%
\thanks{Department of Mathematics, Applied Mathematics, and
  Statistics, Case Western Reserve University, Cleveland, Ohio,
  U.S.A.; mark.meckes@case.edu.}}
\date{\vspace*{-3ex}}

\begin{document}
\sloppy

\maketitle

\begin{abstract}
\noindent
Magnitude is a numerical invariant of compact metric spaces. Its theory is
most mature for spaces satisfying the classical condition of being of
negative type, and the magnitude of such a space lies in the interval $[1,
\infty]$. Until now, no example with magnitude $\infty$ was known. We
construct some, thus answering a question open since
2010. We also give a sufficient condition for the magnitude of a space
to converge to $1$ as it is scaled down to a point, unifying and
generalizing previously known conditions.
\end{abstract}

\section{Introduction}

Magnitude is an invariant defined in the wide generality of enriched
categories and specializing to an invariant of metric
spaces. (See~\cite{MMS}, or \cite{MMSCG} for a survey and~\cite{MAGBIB} for
a bibliography.) It carries abundant geometric information.  For example,
for compact $X \sub \R^N$, consider the function assigning to each $t > 0$
the magnitude of the rescaled space $tX$. The large-scale asymptotics of
this function determine the Minkowski dimension of $X$, its volume, and,
under hypotheses, its surface area (Corollary~7.4 of~\cite{MeckMDC},
Theorem~1 of~\cite{BaCa}, and Theorem~2(d) of~\cite{GiGo}).  Magnitude is
also closely related to certain measures of biodiversity, which themselves
are essentially entropies (\cite{MDBB}, \cite{MEMS} and Chapter~6
of~\cite{ED}).

The definitions are as follows. For a finite metric space $A$, write
$Z_A$ for the matrix $(e^{-d(a, b)})_{a, b \in A}$. If $Z_A$ is
invertible, the \demph{magnitude} $\mg{A}$ of $A$ is the sum of all
the entries of $Z_A^{-1}$. A compact metric space $X$ is
\demph{positive definite} if $Z_A$ is positive definite for all finite
$A \sub X$, and its \demph{magnitude} $\mg{X} \in [0, \infty]$
is then defined as $\sup\{\mg{A} \such \text{finite } A \sub
X\}$. Positive definiteness ensures that this definition is consistent
when $X$ is finite and, as shown in \cite{MeckPDM}, allows the theory
to be developed satisfactorily.

A stronger condition is that $tX$ is positive definite for all
$t > 0$, where $tX$ is shorthand for $X$ equipped with the rescaled
metric $t d_X$; this is equivalent to the classical condition that $X$
is of \demph{negative type} (\cite{MeckPDM}, Theorem~3.3). When $X$ is
a subset of a Banach space equipped with the subspace metric, we can
equivalently instead consider $tX$ to be the usual dilatation of $X$
again equipped with the subspace metric.

Until now, no example was known of a compact positive definite space with
magnitude $\infty$. The question of whether such a space exists was first
raised in the paper~\cite{MeckPDM} (text
preceding Lemma~2.1) posted to the arXiv in 2010 and published in 2013.
It was raised again in~\cite{MeckMDC} (after Definition~3.3), and once
again in~\cite{MMSCG} (as Open Problem~5(1)).

In section \ref{S:infinite} we construct a family of such spaces
$X$. They are moreover of negative type, and we prove not only that
$\mg{X} = \infty$, but also that $\mg{tX} = \infty$ for all $t > 0$.

A complementary question involves the behavior of the magnitude when a
space shrinks to a point. The magnitude of any nonempty positive
definite space lies in the interval $[1, \infty]$, with the lower
bound achieved only by the one-point space. We say that a compact
metric space $X$ has the \demph{one-point property} if
$\lim_{t \to 0^+} \mg{tX} = 1$. Example 2.2.8 of~\cite{MMS}, due to
Willerton, shows that even a finite space of negative type may fail to
have this property. In section~\ref{S:one-point} we prove that a broad
class of compact spaces of negative type do have the one-point
property.  Our result unifies and generalizes some previously known
sufficient conditions, namely that $X$ is isometric to a subset of
$\R^N$ equipped with either the Euclidean metric or the metric induced
by the $1$-norm.

The main tool used to prove both of our main results is Theorem~4.6
in~\cite{MMSCG}, which provides an upper bound, and frequently an exact
formula, for the magnitude of a compact, convex subset of $\ell_1^N$. Here
$\ell_1^N$ denotes $\R^N$ equipped with the metric induced by the
$1$-norm. The new ingredient in both proofs is to
combine the formula for magnitude in $\ell_1^N$ with finite-dimensional
approximations in order to draw conclusions in the infinite-dimensional
spaces $\ell_1$ and $L_1$.

The \demph{$\ell_1$ intrinsic volumes} of a compact, convex set
$A \subseteq \ell_1^N$ are defined by
\[
V_k'(A) = \sum_{1 \leq i_1 < \cdots < i_k \leq N}
\Vol_k\bigl(\pi_{i_1, \ldots, i_k} A\bigr)
\]
where $\pi_{i_1, \ldots, i_k}$ is orthogonal projection onto the subspace
spanned by the standard basis vectors $\vc{e}_{i_1}, \ldots,
\vc{e}_{i_k}$. These quantities were introduced in \cite{IGON}, where it
was shown that there exists a version of integral geometry adapted to the
$1$-norm, with the $\ell_1$ intrinsic volumes playing the role of the
classical intrinsic volumes $V_k$ in Euclidean integral geometry (see
e.g.~\cite{KlRo}).  In fact, this $\ell_1$ integral geometry is valid for
the wider class of $\ell_1$-convex sets (defined in~\cite{IGON}), as are
some of the results in section~\ref{S:one-point} below; but for simplicity,
we state our results for ordinary convex sets only.

The aforementioned Theorem~4.6 in~\cite{MMSCG} is the following.

\begin{thm}
  \label{T:l1-magnitude}
  If $A \subseteq \ell_1^N$ is compact and convex, then
  \begin{equation}
    \label{Eq:l1-magnitude}
    \mg{A} \le \sum_{i=0}^N \frac{1}{2^i} V_i'(A) = \sum_{i=0}^N
    V_i'\bigl(\tfrac{1}{2} A \bigr),
  \end{equation}
  with equality if $A$ has nonempty interior.
\end{thm}

We note that $\sum_{i=0}^N V_i'$ can also be considered an $\ell_1$
analogue of the Wills functional $W = \sum_{i=0}^N
V_i$ (see e.g.~\cite{AGHC}).

\section{Spaces with infinite magnitude}
\label{S:infinite}

As usual, $\ell_1$ denotes the space of real sequences $(x_i)$ whose
1-norm $\sum \mg{x_i}$ is finite, with the metric induced by the
1-norm.
Write $\vc{e}_i$ for the $i$th standard basis vector $(0,
\ldots, 0, 1, 0, \ldots)$ of $\ell_1$ or $\ell_1^N$.

Let $(a_i)$ be a sequence of positive reals converging to $0$, with $\sum
a_i = \infty$. Denote by $X$ the closed convex hull in $\ell_1$ of $\{ a_1
\vc{e}_1, a_2 \vc{e}_2, \ldots\}$, with the subspace metric. Equivalently,
\[
X =
\Bigl\{ (x_1, x_2, \ldots) \such
x_i \geq 0, \ \sum x_i/a_i \leq 1 \Bigr\}.
\]

\begin{thm}
\label{thm:main}
The metric space $X$ is compact and of negative type, and $\mg{tX} =
\infty$ for all $t > 0$.
\end{thm}

Spaces similar to $X$, but with the $\ell_2$ metric, have been studied in
the geometry of Banach spaces (e.g.\ by Ball and Pajor~\cite{BaPa}).

\begin{proof}
First note that
$\lim_{i \to \infty} a_i \vc{e}_i = 0$, which implies that
$\{a_i \vc{e}_i \such i \geq 1\} \cup \{0\}$ is compact and that
its closed convex hull is $X$. But in a Banach space, the closed convex
hull of a compact set is compact (Theorem~5.35 of~\cite{AlBo}), so $X$ is
compact.

That $X$ is of negative type is immediate, since $\ell_1$ is of
negative type (Theorem~3.6(2) of~\cite{MeckPDM}).

It remains to prove that $\mg{tX} = \infty$ for all $t > 0$. Since $tX$ is
of the same form as $X$, we may assume that $t = 1$.

For $N \geq 1$, write $X_N$ for the convex hull of $\{a_1 \vc{e}_1, \ldots,
a_N \vc{e}_N, \vc{0}\}$ in $\ell_1^N$, with the subspace metric.
Theorem~\ref{T:l1-magnitude} implies that
\[
\mg{X_N}
\geq
\hlf V'_1(X_N)
=
\hlf \sum_{i = 1}^N a_i.
\]
Now $\sum_{i = 1}^\infty a_i = \infty$, so $\mg{X_N} \to \infty$ as $N \to
\infty$.

The standard isometry $\ell_1^N \incl \ell_1$ restricts to an isometry $X_N
\to X$ for every $N$. For compact positive definite spaces, magnitude
is monotone with respect to inclusion, so $\mg{X_N} \leq \mg{X}$ for all
$N$. Hence $\mg{X} = \infty$.
\end{proof}

\begin{remark}
If $(a_i)$ is a sequence of positive reals such that $a_i \to 0$ but $\sum
a_i < \infty$, then $\mg{X} < \infty$. Indeed, $X$ is a subspace of the
infinite-dimensional box $Y = \prod_{i = 1}^\infty [0, a_i]$ in $\ell_1$,
so
\[
  \mg{X} \leq \mg{Y} = \prod_{i=1}^\infty \bigl(1 +
  \tfrac{a_i}{2}\bigr) \le e^{\sum a_i/2} < \infty,
\]
as observed in Open Problem~5(1) of~\cite{MMSCG}. Thus, for
$0$-convergent sequences $(a_i)$, the space $X$ has finite magnitude if and
only if the sum $\sum a_i$ is finite.
\end{remark}

Spaces $X$ of the class considered above clearly have the property that if
$X$ has finite magnitude, then its magnitude function is finite for every
$t > 0$.  This latter phenomenon holds in greater generality, as the
following results show.

\begin{propn}
  \label{T:multiple-pd}
  If $A$ is a positive definite compact metric space and $n \in \N$,
  then $nA$ is positive definite and $\mg{nA} \le \mg{A}^n$.
\end{propn}

\begin{proof}
  The map $x \mapsto (x, \dots, x)$ is an isometric embedding
  $nA \hookrightarrow A^n$, where $A^n$ is given the
  $\ell_1$-sum metric.  Therefore $nA$ is positive definite and
  $\mg{nA} \le \mg{A^n} = \mg{A}^n$, by Lemma 3.1.3 and Proposition
  3.1.4 of \cite{MMS}.
\end{proof}

\begin{cor}
  Suppose that $A$ is a compact and convex subset of a Banach space
  and is positive definite. Then $A$ is of negative type, and
  $\mg{A} < \infty$ if and only if $\mg{tA} < \infty$ for every
  $t > 0$.
\end{cor}

\begin{proof}
  By translation we may assume that $0 \in A$, so by convexity $t_1 A
  \subseteq t_2 A$ whenever $0 \le t_1 \le t_2$, and in particular $t
  A \subseteq \lceil t \rceil A$ for every $t > 0$. By Proposition
  \ref{T:multiple-pd}, $\lceil t \rceil A$ is positive definite and
  therefore $tA$ is as well, and furthermore $\mg{tA} \le
  \mg{\lceil t \rceil A} \le \mg{A}^{\lceil t \rceil} < \infty$.
\end{proof}

\section{The one-point property}
\label{S:one-point}

Write $L_1 = L_1[0,1]$ for the Banach space of measurable functions
$f:[0,1] \to \R$ whose integral $1$-norm $\int \abs{f}$ is finite,
with the metric induced by the $1$-norm. We note that a separable
Banach space is a positive definite metric space (equivalently, of
negative type), with the metric induced by its norm, if and only if it
is isometrically isomorphic to a subspace of $L_1$ (Corollary 3.5 in
\cite{MeckPDM}). Examples include both $\ell_1^N$ and $\R^N$ with the
Euclidean metric.

Our second main theorem is the following.
\begin{thm}
  \label{T:L1}
  Suppose $A$ is a nonempty compact subset of a finite-dimensional subspace of
  $L_1$. Then $\mg{A} < \infty$ and $A$ has the one-point
  property.
\end{thm}
The rest of this section is devoted to the proof.

The finiteness statement in Theorem \ref{T:L1} was previously proved
(in a less elementary way) in Proposition~4.13 of~\cite{MMSCG},
following special cases proved earlier in \cite{MMS} and
\cite{MeckPDM}. The one-point property for compact subsets of
$\ell_1^N$ was first explicitly noted in Proposition~4.4 of
\cite{MMSCG} (but follows easily from results in
\cite{MMS}). Independent proofs of the one-point property for subsets
of $\R^N$ were given in \cite{BaCa,WillMOBH,MeckMIV}. Theorem
\ref{T:L1} simultaneously generalizes these facts. (In \cite{MeckMIV},
it was further proved that so-called GB-bodies in a Hilbert space have
finite magnitude and the one-point property, with a proof closely
related to the proof of Theorem \ref{T:L1}.)

The proof of Theorem \ref{T:L1} has three main ingredients: a
classical approximation procedure that allows us to reduce
consideration to subspaces of $\ell_1^N$, a bound on magnitude in
terms of $V_1'$ which follows from Theorem \ref{T:l1-magnitude}, and a
dimension-independent bound on $V_1'$ for polytopes. Here, a
\demph{polytope} is the convex hull of a finite set.

\begin{lemma}
  \label{T:l1-concavity}
  If $A \subseteq \ell_1^N$ is convex, then
  \[
    (j+k)!V_{j+k}'(A) \le \bigl(j!V_j'(A) \bigr) \bigl(k!V_k'(A) \bigr)
  \]
  for each $j, k \ge 0$.
\end{lemma}

\begin{proof}
  By the definition of the $\ell_1$ intrinsic volumes,
  \[
    (j+k)! V_{j+k}'(A) = \sum_{i_1, \dots, i_{j+k} = 1}^N
                  \Vol_{j+k} (\pi_{i_1, \dots, i_{j+k}}(A)),
  \]
  noting that if $i_1, \ldots, i_{j + k}$ are not all distinct then
  the corresponding summand vanishes. Hence
  \begin{align*}
    (j+k)! V_{j+k}'(A)
    & \le \sum_{i_1, \dots, i_{j+k} = 1}^N \Vol_{j+k}
      \bigl(\pi_{i_1, \dots, i_j}(A) \times
    \pi_{i_{j+1}, \dots, i_{j+k}}(A)\bigr) \\
    & = \bigl(j!V_j'(A) \bigr) \bigl(k!V_k'(A) \bigr).
  \end{align*}
\end{proof}

The $j=1$ case of Lemma \ref{T:l1-concavity} implies the following
result by induction.

\begin{propn}
  \label{T:l1-McMullen}
  If $A \subseteq \ell_1^N$ is compact and convex, then
  \[
    V_k'(A) \le \frac{1}{k!} V_1'(A)^k
  \]
  for each $0 \le k \le N$.
\end{propn}

Combining Proposition \ref{T:l1-McMullen} and Theorem \ref{T:l1-magnitude}
we obtain the following.

\begin{cor}
  \label{T:l1-mag-V1}
  If $A \subseteq \ell_1^N$ is compact and convex, then
  \[
    \mg{A} \le \exp(V_1'(\tfrac{1}{2}A)).
  \]
\end{cor}

\begin{remark}
  For the classical intrinsic volumes $V_k$, the estimate
  $V_k \le \tfrac{1}{k!} V_1^k$ analogous to
  Proposition~\ref{T:l1-McMullen} was independently derived by Chevet
  (Lemme~4.2 in~\cite{Chev}) and McMullen (Theorem~2
  in~\cite{McMuIBI}) from the Alexandrov--Fenchel inequalities. As
  noted by McMullen, this implies the bound $W \le \exp(V_1)$ on the
  Wills functional $W = \sum V_k$, analogous to Corollary
  \ref{T:l1-mag-V1}.
\end{remark}

\begin{lemma}
  \label{T:polytope-V1}
  Suppose that $P \subseteq \ell_1^N$ is a polytope with $m$
  vertices.  Then
  \[
    V_1'(P) \le 2(m-1) \diam(P),
  \]
  where $\diam(P)$ is the diameter of $P$ in the $\ell_1$ metric.
\end{lemma}

\begin{proof}
  By translation, we may assume that one of the vertices of $P$ is at
  the origin.  We write $P = \conv \{v_1, \ldots, v_{m-1}, 0 \}$, set
  $v_m = 0$, and denote $v_k = (v_k(1), \dots, v_k(N))$. Then
  \begin{align*}
    V_1'(P) & = \sum_{i=1}^N \bigl( \max_k v_k(i) - \min_k
              v_k(i)\bigr) \le \sum_{i=1}^N 2 \max_k \abs{v_k(i)} \\
    & \le 2 \sum_{i=1}^N \sum_{k=1}^{m-1} \abs{v_k(i)}
      = 2 \sum_{k=1}^{m-1} \norm{v_k}_1 \le 2(m-1) \diam(P).
  \end{align*}
\end{proof}

Corollary \ref{T:l1-mag-V1} and Lemma \ref{T:polytope-V1}
immediately imply the following.

\begin{cor}
  \label{T:l1-mag-diam}
  If $P \subseteq \ell_1^N$ is a polytope with $m$
  vertices, then
  \begin{equation}
    \label{Eq:l1-mag-diam}
    \mg{P} \le \exp((m-1) \diam(P)).
  \end{equation}
\end{cor}

\begin{remark}
Theorem 6.2 of~\cite{IGON} and Proposition~\ref{T:l1-McMullen} together
imply that
\[
\Vol_N\bigl(P + [0, 1]^N\bigr)
=
\sum_{k = 0}^N V'_k(P)
\leq
\exp(V'_1(P))
\]
for every polytope $P \sub \ell_1^N$. It follows from
Lemma~\ref{T:polytope-V1} that
\[
  \Vol_N\bigl(P + [0,1]^N\bigr) \le \exp(2 (m-1) \diam(P)),
\]
where $m$ is the number of vertices and the diameter is in the $\ell_1$
metric. Despite the classical flavor of this estimate we have not seen it
stated elsewhere.
\end{remark}

\begin{cor}
  \label{T:L1-polytope}
  If $P \subseteq L_1$ is the convex hull of $m$ points, then
  \[
    \mg{P} \le \exp((m-1) \diam(P)).
  \]
\end{cor}

\begin{proof}
  Let $E \subseteq L_1$ be the linear span of $P$.  It is well known
  (e.g.~\cite{TalaESL}, section~1) that $E$ can be approximated in the
  Banach--Mazur distance by a sequence of subspaces $E_n \subseteq
  \ell_1^{N_n}$.  It follows (as in section~3.A of~\cite{GromMSR}) that $P$
  is the limit, in the Gromov--Hausdorff distance, of a sequence of
  polytopes $P_n \subseteq \ell_1^{N_n}$, each with at most $m$ vertices.

  The magnitude of compact positive definite metric spaces is
  lower semicontinuous with respect to the Gromov--Hausdorff
  distance \cite[Theorem 2.6]{MeckPDM}, and diameter is continuous. By
  Corollary \ref{T:l1-mag-diam} we therefore have
  \[
    \mg{P} \le \liminf_{n \to \infty} \mg{P_n} \le \liminf_{n\to
      \infty} e^{(m-1) \diam(P_n)} = e^{(m-1) \diam(P)}.
  \]
\end{proof}

\begin{pfof}{Theorem \ref{T:L1}}
  Let $P$ be a polytope lying in the linear span of $A$ and
  containing $A$.  Then for each $t > 0$,
  \[
    1 \le \mg{tA} \le \mg{tP} \le \exp\bigl((m-1)t\diam(P)\bigr),
  \]
  where $m$ is the number of vertices of $P$. The theorem follows.
\end{pfof}

\paragraph{Acknowledgements} TL was supported in part by a
Leverhulme Research Fellowship. MM's research was supported in part by
Collaboration Grant 315593 from the Simons Foundation.  This work was
partly done while MM was visiting the Mathematical Institute of the
University of Oxford, partially supported by ERC Advanced Grant 740900
(LogCorRM) to Prof.\ Jon Keating and Simons Fellowship 678148 to Elizabeth
Meckes.  MM thanks the Institute and Prof.\ Keating for their hospitality.

\bibliography{mathrefs}

\end{document}

%% file: ab.tex
\usepackage{amssymb,amsmath}
\usepackage{color}

\newcommand{\such}{:}
\newcommand{\sub}{\subseteq}
\newcommand{\reals}{\mathbb{R}}
\newcommand{\demph}[1]{\textbf{\textup{#1}}}
\newcommand{\mg}[1]{\lvert #1 \rvert}
\newcommand{\R}{\reals}
\newcommand{\vc}[1]{#1} 
\newcommand{\hlf}{\tfrac{1}{2}}
\DeclareMathOperator{\Vol}{Vol}
\newcommand{\incl}{\hookrightarrow}
\newcommand{\N}{\mathbb{N}}
\DeclareMathOperator{\conv}{conv}
\DeclareMathOperator{\diam}{diam}
\newcommand{\abs}[1]{\left\vert #1 \right\vert}
\newcommand{\norm}[1]{\left\Vert #1 \right\Vert}

%% file: thmlist.tex
\usepackage[amsmath,thmmarks]{ntheorem}

\newtheorem{thm}{Theorem}[section]
\newtheorem{propn}[thm]{Proposition}
\newtheorem{lemma}[thm]{Lemma}
\newtheorem{cor}[thm]{Corollary}

\theorembodyfont{\normalfont}

\newtheorem{remark}[thm]{Remark}

\theoremstyle{nonumberplain}
\theoremsymbol{\ensuremath{\square}}
\qedsymbol{\ensuremath{\square}}

\newtheorem{proof}{Proof}

\newcommand{\theoremtobeproved}{}
\newtheorem{pfoftheorem}{Proof of \theoremtobeproved}
\newenvironment{pfof}[1]
{
\renewcommand{\theoremtobeproved}{#1}
\begin{pfoftheorem}
}
{\end{pfoftheorem}}